\documentclass[11pt,letterpaper]{article}

\usepackage[T2A,OT1]{fontenc}
\usepackage[utf8]{inputenc}
\usepackage[russian,english]{babel}
\usepackage[textwidth=360pt,textheight=584pt,centering]{geometry}
\usepackage{amsmath,amssymb,amsthm,mathtools}
\usepackage[hidelinks]{hyperref}
\theoremstyle{plain}
\newtheorem{theorem}{Theorem}[section]
\newtheorem{lemma}[theorem]{Lemma}
\newtheorem{proposition}[theorem]{Proposition}

\theoremstyle{definition}
\newtheorem{definition}[theorem]{Definition}

\theoremstyle{remark}

\title{On omega-powers of infinite difference rank}
\author{M. Vishnikin}
\date{\today}

\begin{document}

\maketitle

\begin{abstract}
It is proved that there exists a language $L$ such that its $\omega$-power has infinite rank in the difference hierarchy of open sets. The same holds for the decreasing difference hierarchy of $\Pi^0_2$ sets. This answers the questions asked in \cite[Section~4]{FL24} and \cite[Section~4]{FGL26}.
\end{abstract}

\section{Introduction}
Regular languages of $\omega$-words play an important role in theoretical computer science. They were introduced by B\"uchi to prove the decidability of the monadic second-order theory of $(\mathbb{N},S)$, where $S$ is the successor relation~\cite{Buc62}.

It is known that $\omega$-languages accepted by Muller automata are naturally ordered in a hierarchy called the Wagner hierarchy~\cite{Wag79}. This hierarchy is related to descriptive set theory and coincides with the Wadge hierarchy restricted to this class of languages.

Nevertheless, the strict subclass of regular $\omega$-languages obtained from regular languages of finite words by taking their $\omega$-powers is also of interest. Namely, let $L$ be a language, that is, a subset of $\Sigma^*$, the set of all finite words over the finite alphabet $\Sigma$. We can then consider the language $L^\omega$, defined as the set of $\omega$-words (that is, words whose letters are indexed by the ordinal $\omega$) formed by concatenating words belonging to $L$. We use the notation $L^\omega$, as in~\cite{FGL26}, although the notation $L^{\infty}$ is also used, for example in~\cite{FL24}.

This paper answers questions posed in~\cite[Section~4]{FL24} and~\cite[Section~4]{FGL26}, where the $\Delta^0_2$ and $\Delta^0_3$ levels of $\omega$-powers of regular languages were studied. Namely, we show how to construct two $\omega$-powers in the classes $\Delta^0_2$ and $\Delta^0_3$ that do not belong to any finite level of the difference hierarchy over $\Sigma^0_1$ and the decreasing difference hierarchy over $\Pi^0_2$, respectively.

The constructed $\omega$-powers are not regular languages. Indeed, its well known that every regular $\omega$-language in $\Delta^0_2$ belongs to a finite level of the difference hierarchy over $\Sigma^0_1$, and every regular $\omega$-language belongs to a finite level of the decreasing difference hierarchy over $\Pi^0_2$. Consequently, the underlying finitary languages are also non-regular, since the $\omega$-power of a regular finitary language is regular.

Thus, the main results of this paper are the following theorems.

\begin{theorem}\label{thm:main2}
There exists a language $L \subseteq\{0,1\}^*$ such that $L^\omega \in \Delta^0_2$ and $L^\omega \notin \bigcup_{n\in \omega} D_n(\Sigma^0_1)$.
\end{theorem}

\begin{theorem}\label{thm:main3}
There exists a language $L \subseteq\{0,1\}^*$ such that $L^\omega \in \Delta^0_3$ and $L^\omega \notin \bigcup_{n\in \omega} D^\downarrow_n(\Pi^0_2)$.
\end{theorem}

\section{Preliminaries}

We use standard notation from descriptive set theory~\cite{Kec95}. For clarity, we briefly recall the basic definitions.

Let $\Sigma$ be a finite alphabet. In this paper, we are interested in the alphabet $\Sigma=\{0,1\}$, which we denote by $2$ for convenience. We write $2^\omega$ for the set of words of length $\omega$ over this alphabet, $2^*$ for the set of finite words and $2^+$ for the set of nonempty words. The concatenation of finite words $u$ and $v$ is denoted by $uv$. A finite word $u$ is a prefix of a finite or infinite word $v$, denoted by $u \leq v$, if there is a word $w$ such that $v=uw$.

The space $2^\omega$ carries the product topology, whose basic open sets are the cylinders $N_u := \{x \in 2^\omega \mid u \leq x\}$, where $u\in 2^*$. We now define the Borel hierarchy on this space.

Let $\Sigma^0_1$ be the class of open sets, that is,
\[
\Sigma^0_1 := \left\{\bigcup_{u\in S}N_u \mid S\subseteq 2^*\right\}.
\]
For each nonzero countable ordinal $\xi$, we put $\Pi^0_\xi := \{2^\omega\setminus U \mid U\in\Sigma^0_\xi\}$ and $\Delta^0_\xi := \Sigma^0_\xi\cap\Pi^0_\xi$. For $2\leq\xi<\omega_1$, the classes $\Sigma^0_\xi$ are defined inductively by
\[
\Sigma^0_\xi := \left\{\bigcup_{n\in\omega}A_n \mid \forall n\in\omega\; A_n\in\bigcup_{1\leq\eta<\xi}\Pi^0_\eta\right\}.
\]

We use two versions of the difference hierarchy. The increasing version is used to define the Lavrentieff hierarchy in~\cite[Section~2]{FL24}.

\begin{definition}[Difference hierarchy]\label{def:diff}
Fix a countable ordinal $\zeta$ and a set $X$. Let $O=(O_\eta)_{\eta<\zeta}$ be a family of subsets of $X$ with $O_\eta\subseteq O_\theta$ whenever $\eta<\theta<\zeta$. Put
\[
\begin{aligned}
D_\zeta(O):=\{x\in X\mid \exists\eta<\zeta\,
(&\operatorname{parity}(\eta)\neq\operatorname{parity}(\zeta)\\
&{}\land x\in O_\eta\setminus\bigcup_{\theta<\eta}O_\theta)\}.
\end{aligned}
\]
Here the parity of an ordinal $\alpha=\lambda+n$, where $\lambda$ is zero or a limit ordinal and $n\in\omega$, is the parity of $n$.

For a Polish space $X$ and a class $\Gamma$ of subsets of Polish spaces, define
\[
D_\zeta(\Gamma)(X):=\{D_\zeta(O)\mid
\forall\eta<\zeta\; O_\eta\in\Gamma(X)\},
\]
where $O$ ranges over increasing families as above. In particular, $D_0(\Gamma)(X)=\{\emptyset\}$. Taking $\Gamma=\Sigma^0_\xi$ gives the difference hierarchy over $\Sigma^0_\xi$ sets.
\end{definition}

We also use the decreasing version of the difference hierarchy, extending the finite-level definition in~\cite[Definition~6]{FGL26}.

\begin{definition}[Decreasing difference hierarchy]\label{def:diff-down} Fix a countable ordinal $\zeta$ and a set $X$. Let $G=(G_\eta)_{\eta<\zeta}$ be a family of subsets of $X$ with $G_\theta\subseteq G_\eta$ whenever $\eta<\theta<\zeta$, and put $G_\zeta=\emptyset$. Define
\[
D^\downarrow_\zeta(G):=
\bigcup_{\substack{\eta<\zeta\\\operatorname{parity}(\eta)=0}}
(G_\eta\setminus G_{\eta+1}).
\]
For a Polish space $X$ and a class $\Gamma$ of subsets of Polish spaces, put
\[
D^\downarrow_\zeta(\Gamma)(X):=\{D^\downarrow_\zeta(G)\mid
\forall\eta<\zeta\;G_\eta\in\Gamma(X)\},
\]
where $G$ ranges over decreasing families as above. Taking $\Gamma=\Pi^0_2$ gives the hierarchy used in~\cite{FGL26}. For $n<\omega$, reversing the finite family gives $D^\downarrow_n(\Gamma)(X)=D_n(\Gamma)(X)$.
\end{definition}

\begin{definition}\label{def:reduction}
Let $X$ and $Y$ be two topological spaces. We say that a set $A\subseteq X$ continuously reduces to $B\subseteq Y$ if there is a continuous function $f:X\to Y$ such that the preimage $f^{-1}(B)=\{x\in X\mid f(x)\in B\}$ is equal to $A$ (i.e., $x\in A\iff f(x)\in B$ for every $x\in X$).
\end{definition}

The levels of the Borel hierarchy and of both versions of the difference hierarchy ($D_\eta$ and $D^\downarrow_\eta$) are closed under continuous preimages. Moreover, the finite difference hierarchies over $\Sigma^0_1$ and $\Pi^0_2$ are strict.

\section{The proof of the \texorpdfstring{$\Delta^0_2$}{Delta-0-2} case}

First, let us recall the result in~\cite{FL24}. We will only need the following theorem. 

\begin{theorem}[Finkel--Lecomte]\label{thm:FL24}
For any $n$, there exists a regular language $L \subseteq 2^{+}$ such that $L^{\omega}$ is $D_n(\Sigma^0_1)$-complete. 
\end{theorem}

Now let us fix the sequence of languages obtained from Theorem~\ref{thm:FL24}. That is, we have languages $\{L_i\}_{i \geq 1}$ such that each $L_i^\omega$ is complete for the corresponding class $D_i(\Sigma^0_1)$, and each $L_i$ is a regular language defined over the alphabet $\{0,1\}$.

Now the main idea is to define a language $A$ for which the topological complexity of $A^\omega$ will be greater than the topological complexity of each $L_i^\omega$. To do this, we need to encode all languages $\{L_i\}_{i \geq 1}$ as disjoint subsets of $2^*$. Each letter in a word of $L_i$ is encoded by a string determined by that letter and the index $i$. More formally, we want to define different strings $h_i(0)$ and $h_i(1)$ which will be the encodings of $0$ and $1$ when we have a word from the language $L_i$. Since we need the codes not to intersect, we need the codewords $h_i(0)$ and $h_i(1)$, $i\geq 1$, to form a prefix code. 

It is obvious that this can be done in different ways, and the easiest way is the following. 

\begin{definition}\label{def:code}
For each natural number $n \geq 1$, we define 
$$ h_n(0) = 1 0^{2n-1} 1 , \quad h_n(1) = 1 0^{2n} 1$$
\end{definition}

Each $h_n$ extends naturally to finite and infinite words and then to languages as follows: 
\begin{align*}
& h_n(L) := \{ h_n(a_0) \ldots h_n(a_k) \mid a_0 \ldots a_k \in L\}; \\
& h_n(L^\omega) := \{ h_n(w_0) h_n(a_1) \ldots \mid \forall i\in \omega w_i \in L \}.
\end{align*}

The next step in the proof is to collect all the words obtained by encoding each $L_i$. We define a language $A_0$ which consists of all homomorphic encodings of words in the languages $L_i$. More formally, the language $A_0$ is defined as follows: 
$$A_0 = \bigcup_{i=1}^{\infty} h_i(L_i).$$

So it is not hard to see that $L_i^{\omega} \leq_W A_0^\omega$ for each index. But the language $A_0$ cannot be a suitable example for the main question. The main problem is that $A_0^\omega$ could be a language not in $\Delta^0_2$. 

Moreover, it is not hard to see that, for some collection of topologically simple languages $L_i^\omega$ and homomorphisms $h_i$, the language $A_0^\omega$ can be $\Pi^0_2$-complete. For example, suppose that all $L_i$ are equal to the language consisting of the single word $0$ and that $h_i (0) = 0^{i-1} 1$. Then $A_0^\omega$ will be the language that consists of words with infinitely many occurrences of $1$. So $A_0^\omega$ is a classical example of a $\Pi^0_2$ set.

So the next idea is to modify the language $A_0$ and to define a language $A$ whose $\omega$-power will be a $\Delta^0_2$ set. We define $A$ in the following way:
$$A = A_0 \cup \{x \in 2^* \mid \forall i \geq 1 \forall y \in h_i(2^\omega) (\neg(x \leq y))\}.$$ 
So we have just added all bad prefixes to $A_0$, where ``bad'' means that such a word cannot be a prefix of any word $h_i(w)$. Note that this also means that such a word cannot be a prefix of any word in $A_0$. 

Now, first of all, we need to clarify what $A^\omega$ is. The next proposition does this. 
\begin{proposition}\label{prop:form}
$A^\omega = \bigcup_{i = 1}^{\infty} (h_i(L_i^\omega)) \cup (2^\omega \setminus (\bigcup_{i=1}^{\infty}(h_i(2^\omega)) \cup \{10^\omega\}))$.
\end{proposition}
\begin{proof}
We prove the equality in both directions.
	
($\supseteq$) First, let us show that each $h_i(L_i^\omega)$ is a subset of $A^\omega$. We note that the set $h_i(L_i)$ is a subset of $A$. 

Next, if we have a word $h_i(x) \in h_i(L_i^\omega)$, then $x = x_0 x_1 \ldots$, where each $x_j$ is an element of $L_i$. But each $h_i(x_j)$ is an element of $h_i(L_i)$, and hence the word $h_i(x) = h_i(x_0) h_i(x_1) \ldots$ belongs to $(h_i(L_i))^\omega$. So we have proved that $\bigcup_{i = 1}^{\infty} (h_i(L_i^\omega)) \subseteq A^\omega$.

Now, assume that $x$ is an element of $(2^\omega \setminus (\bigcup_{i=1}^{\infty}(h_i(2^\omega)) \cup \{10^\omega\}))$. There are two possible cases. 

First, suppose that $x$ cannot be partitioned into blocks of the form $1 0^i 1$ (where $i \geq 1$). We separate finitely many such blocks, obtaining $x = x_0 x_1 \ldots x_n \alpha$, until we cannot find the next block. This means that $\alpha$ starts with $0$ or $11$, or is equal to $1 0^\omega$. 

Suppose that $\alpha = 0 \beta$ or $ \alpha = 11 \beta$, where $\beta \in 2^\omega$. If $\beta \in (2^\omega \setminus (\bigcup_{i=1}^{\infty}(h_i(2^\omega)) \cup \{10^\omega\}))$, we can apply the same argument to $\beta$, and the first part of the partition of $x$ will be $x_0 \ldots x_n 0$ (or $x_0 \ldots x_n 11$ in the second case), which is an element of $A$ because no correct homomorphic images have such a prefix. If $\beta \in h_j(2^\omega)$ for some $j$, then $\beta = y_0 y_1 \ldots$. If $\alpha$ starts with $0$, then the correct partition of $x$ is $(x_0 \ldots x_n 0 y'_0)(y''_0 y'_1) (y''_1 y'_2) \ldots $, where $y'_{j}$ is the part of $y_{j}$ ending just after the first $0$, and $y''_{j}$ is the remaining part. If $\alpha$ starts with $11$, the partition is $(x_0 \ldots x_n 11 y'_0)(y''_0 y'_1) (y''_1 y'_2) \ldots $. Finally, if $\beta = 10^\omega$, the partition of $x$ is $(x_0 \ldots x_n 0 1) (0) (0) \ldots$ if $\alpha$ starts with $0$, or $(x_0 \ldots x_n 1 1 1) (0) (0) \ldots$ if $\alpha$ starts with $11$.

If $\alpha = 1 0^\omega$, then the partition of $x$ is $(x_0 \ldots x_n 1 0^{m+2}) (0) (0) \ldots$, where the number $m$ is chosen as the maximum number of occurrences of $0$ in the blocks $x_j$. The only possibility for this string not to be in $A^\omega$ is that there are no $x_i$ at the beginning of the string, but $x$ cannot be $10^\omega$.

Second, suppose that $x$ can be partitioned into blocks of the form $1 0^i 1$ (where $i \geq 1$), but there is no single $n$ such that all these blocks have one of the forms $1 0^{2n-1} 1$ and $1 0^{2n} 1$. So let $x = x_0 \ldots x_k \ldots x_m x_{m+1} \ldots$, where $x_k$ and $x_m$ are blocks from different homomorphic images. Now we can partition $x$ into the following blocks: $ (x_0 \ldots x_m x'_{m+1}) (x''_{m+1} x'_{m+2}) (x''_{m+2} x'_{m+3}) \ldots $ (parentheses indicate the partition), where $x'_{j}$ is the part of $x_{j}$ ending just after the first $0$, and $x''_{j}$ is the remaining part. Each part of this partition belongs to $A$ because no correctly encoded word can start with it.

($\subseteq$) Let us consider a word $x = x_0 x_1 \ldots$ in $A^\omega$. If $x$ does not belong to $\bigcup_{i=1}^{\infty}(h_i(2^\omega))$, then, of course, it is an element of $2^\omega \setminus (\bigcup_{i=1}^{\infty}(h_i(2^\omega)) \cup \{10^\omega\})$, since $x$ cannot be $10^\omega$. 

So we can consider a word $x$ that belongs to $h_j(2^\omega)$ for some $j$. So $x_0$ belongs to $A_0$ because it is a prefix of some word $h_j(z)$. So $x_0 = h_j(y_0) \in h_j(L_j)$, and the word $x_1 x_2 \ldots$ is also an element of $A^\omega$ and also belongs to $h_j(2^\omega)$. So, after infinitely many steps (considering $x_j x_{j+1} \ldots$ and constructing $y_j$), we get a word $y = y_0 y_1 \ldots$ which belongs to $L_j^\omega$, and $h_j(y) = x_0 x_1 \ldots$, so $x \in h_j(L_j^\omega)$.
\end{proof}

Now we know how the language $A^\omega$ can be represented. The next step is to show that this language is topologically quite simple. 

\begin{proposition}\label{prop:delta2}
$A^\omega \in \Delta^0_2$.
\end{proposition}
\begin{proof}
First, we prove that $A^\omega \in \Sigma^0_2$ and then that $A^\omega \in \Pi^0_2$, which is equivalent to $2^\omega\setminus A^\omega \in \Sigma^0_2$. We first need the following quite simple lemma.
\begin{lemma}\label{lem:images}
	If $h$ is an injective homomorphism $2^\omega \to 2^\omega$, then the following hold: 
	\begin{itemize}
	\item if $A \in \Pi^0_1$, then $h(A) \in \Pi^0_1$;
	\item if $A\in \Sigma^0_2$, then $h(A) \in \Sigma^0_2$;
	\item if $A\in \Pi^0_2$, then $h(A) \in \Pi^0_2$;
	\item if $A\in \Sigma^0_3$, then $h(A) \in \Sigma^0_3$.
	\end{itemize}
\end{lemma}
\begin{proof}
The first statement is obvious: any limit point of $h(A)$ has to be in $h(A)$ because its inverse image has to be in $A$. The second statement follows directly from the first one. The third statement follows from the representation $h(A) = h(2^\omega)\setminus h(2^\omega \setminus A) = h(2^\omega) \cap (2^\omega\setminus h(2^\omega \setminus A))$, where the first set belongs to $\Pi^0_1$ and hence to $\Pi^0_2$, and the set $h(2^\omega \setminus A)$ belongs to $\Sigma^0_2$, so its complement belongs to $\Pi^0_2$. The fourth statement follows immediately from the third.
\end{proof}

($\Sigma^0_2$) Each $L_i^\omega$ is a $\Sigma^0_2$ set since it is $D_i(\Sigma^0_1)$-complete. So, by Lemma~\ref{lem:images}, each $h_i(L_i^\omega)$ is a $\Sigma^0_2$ set. 

Next, we need to analyze $\bigcup_{i=1}^{\infty} h_i(2^\omega) \cup \{10^\omega\}$. The task now is to prove that this set belongs to $\Pi^0_2$ and hence that its complement belongs to $\Sigma^0_2$. 

Let us fix a sequence $a_n = h_{i(n)}(b_n)$ which converges to some $a$. We have two possible cases. First, if the set $\{i(n)\}$ is finite, then, for some $n$, all $a_m$ with $n<m$ are obtained by the same homomorphism $h_n$. So the element $a$ has to be an element of $h_n(2^\omega)$. Second, if the set $\{i(n)\}$ is infinite, then the only element to which $a_n$ can converge is $10^\omega$. 

So the set $\bigcup_{i=1}^{\infty} h_i(2^\omega) \cup \{10^\omega\}$ is equal to $cl(\bigcup_{i=1}^{\infty} h_i(2^\omega) \cup \{10^\omega\})$. So the whole set belongs to $\Pi^0_1$, and hence its complement belongs to $\Sigma^0_2$. 

Finally, we have shown that $A^\omega$ is a $\Sigma^0_2$ set, as it is a countable union of $\Sigma^0_2$ sets. 

($\Pi^0_2$) The set $2^\omega\setminus A^\omega$ consists of words $w$ such that $w \in h_j(2^\omega)$ but $w \notin h_j(L_j^\omega)$ for some $j$, or such that $w$ is exactly the word $10^\omega$. So it can be represented as follows: $\bigcup_{i=1}^{\infty}(h_i(2^\omega) \setminus h_i(L_i^\omega)) \cup \{10^\omega\} =\bigcup_{i=1}^{\infty}(h_i(2^\omega \setminus L_i^\omega)) \cup \{10^\omega\} $. Each set $2^\omega \setminus L_i^\omega$ is a $\Delta^0_2$ set because $L_i^\omega$ is a $D_i(\Sigma^0_1)$-complete set by construction and Theorem~\ref{thm:FL24}. So, by Lemma~\ref{lem:images}, each set in the union is a $\Sigma^0_2$ set; hence the union is also a $\Sigma^0_2$ set. The set $\{10^\omega\}$ belongs to $\Pi^0_1$, and hence the whole union belongs to $\Sigma^0_2$.
\end{proof}

\begin{proposition}\label{prop:rank2}
$A^\omega \notin \bigcup_{n\in \omega} D_n(\Sigma^0_1)$.
\end{proposition}
\begin{proof}
The proof is by contradiction. Suppose that $A^\omega \in D_k(\Sigma^0_1)$. By the definition of $A$, we can consider a language $L_i^\omega$ with $i > k$. Then consider $h_i^{-1}(A^\omega)$, which is equal to $L_i^\omega$ by Proposition~\ref{prop:form}. Since $h_i$ is continuous (which can be seen directly), $h_i^{-1}(A^\omega) = L_i^\omega$ must belong to $D_k(\Sigma^0_1)$, but it does not, since $L_i^\omega$ is $D_i(\Sigma^0_1)$-complete by Theorem~\ref{thm:FL24}.
\end{proof}

This construction yields the following theorem, which answers the question posed in~\cite[Section~4]{FL24}.

\begin{theorem}[Theorem~\ref{thm:main2}, restated]\label{thm:result2}
There exists a language $L \subseteq\{0,1\}^*$ such that $L^\omega \in \Delta^0_2$ and $L^\omega \notin \bigcup_{n\in \omega} D_n(\Sigma^0_1)$.
\end{theorem}

\section{The proof of the \texorpdfstring{$\Delta^0_3$}{Delta-0-3} case}

The proof is essentially the same as in the case of $\Delta^0_2$. The only difference is that we need to use another sequence of languages. First, we recall the following result from~\cite[Theorem~14]{FGL26}.

\begin{theorem}[Finkel--Girardot--Lecomte]\label{thm:FGL26}
For any $n$, there exists a regular language $L \subseteq 2^{+}$ such that $L^{\omega}$ is $D^\downarrow_{2n+1}(\Pi^0_2)$-complete. 
\end{theorem}

We now choose a new sequence of languages $L_i$ obtained from Theorem~\ref{thm:FGL26}. The encodings remain the same. So all the definitions of $h_i$, $A$, and $A_0$ also remain the same. The main representation in Proposition~\ref{prop:form} also holds, since it concerns only the encodings. 

So the first change concerns Proposition~\ref{prop:delta2}; now we need to prove that $A^\omega$ is a $\Delta^0_3$ set. 

\begin{proposition}\label{prop:delta3}
$A^\omega \in \Delta^0_3$.
\end{proposition}
\begin{proof}
Again, we need to prove that this set and its complement belong to $\Sigma^0_3$. Let us recall Proposition~\ref{prop:form}: 
$$A^\omega = \bigcup_{i = 1}^{\infty} (h_i(L_i^\omega)) \cup (2^\omega \setminus (\bigcup_{i=1}^{\infty}(h_i(2^\omega))\cup \{10^\omega\})).$$ 

($\Sigma^0_3$) Each $L_i^\omega$ is a $\Sigma^0_3$ set since it belongs to $D^\downarrow_{2i+1}(\Pi^0_2)$. So, for each $i$, the set $h_i(L_i^\omega)$ belongs to $\Sigma^0_3$ by Lemma~\ref{lem:images}. The set $2^\omega \setminus (\bigcup_{i=1}^{\infty}(h_i(2^\omega))\cup \{10^\omega\})$ belongs to $\Sigma^0_2$, since it is the same set as in the proof of Proposition~\ref{prop:delta2}. So $A^\omega$ is a $\Sigma^0_3$ set.

($\Pi^0_3$) We have $2^\omega\setminus A^\omega = \bigcup_{i=1}^{\infty}(h_i(2^\omega \setminus L_i^\omega)) \cup \{10^\omega\}$. Each set $2^\omega \setminus L_i^\omega$ is a $\Delta^0_3$ set, and hence $h_i(2^\omega \setminus L_i^\omega)$ is a $\Sigma^0_3$ set by Lemma~\ref{lem:images}. So the whole set $2^\omega \setminus A^\omega$ is $\Sigma^0_3$, because the singleton $\{10^\omega\}$ belongs to $\Pi^0_1$.
\end{proof}

\begin{proposition}\label{prop:rank3}
$A^\omega \notin \bigcup_{n\in \omega} D^\downarrow_n(\Pi^0_2)$.
\end{proposition}
\begin{proof}
The same argument as in the proof of Proposition~\ref{prop:rank2} applies. Consider $h_i^{-1}(A^\omega)$ with an appropriate $i$.
\end{proof}

Finally, we obtain the following theorem, which answers the question posed in~\cite[Section~4]{FGL26}.

\begin{theorem}[Theorem~\ref{thm:main3}, restated]\label{thm:result3}
There exists a language $L \subseteq\{0,1\}^*$ such that $L^\omega \in \Delta^0_3$ and $L^\omega \notin \bigcup_{n\in \omega} D^\downarrow_n(\Pi^0_2)$.
\end{theorem}

\end{document}